\documentclass[11 pt]{amsart}
\usepackage[english]{babel}
\usepackage{amsmath}
\usepackage{amstext}
\usepackage{amssymb}
\usepackage{amsthm}
\usepackage{amscd}
\usepackage{amsfonts}
\usepackage[utf8]{inputenc}      
\usepackage{graphicx}

\renewcommand{\footnote}{\endnote}
\newtheorem{theorem}{Theorem}
 \newtheorem{lemma}[theorem]{Lemma}
\newtheorem{proposition}[theorem]{Proposition}
\newtheorem*{definition}{Definition}

\newtheorem{remark}{Remark}
\newtheorem{corollary}[theorem]{Corollary}
\begin{document}
\title{Periodic billiard trajectories in polyhedra}
\author{Bedaride nicolas}
\address{Fédération de recherche des unités de mathématiques de Marseille,
Laboratoire d'analyse, topologie et probabilités, UMR 6632 ,
Avenue Escadrille Normandie Niemen 13397 Marseille cedex 20,
France}

\email{nicolas.bedaride@univ-cezanne.fr}
\date{}
\begin{abstract}
We consider the billiard map inside a polyhedron.
We give a condition for the stability of the periodic
trajectories. We apply this result to the case of
the tetrahedron. We deduce  the existence of an open set of tetrahedra
which have a periodic orbit of length four (generalization of Fagnano's
orbit for triangles), moreover we can study completly the orbit of points
along this coding.
\end{abstract}
\maketitle
\bibliographystyle{plain}

\section{Introduction}
We consider the billiard problem inside polyhedron. We start with
a point of the boundary of the polyhedron and we move along a
straight line until we reach the boundary, where there is
reflection according to the mirror law. A famous example of a
periodic trajectory is Fagnano's orbit: we consider an acute
triangle and the foot points of the altitudes. Those points form a
billiard trajectory which is periodic \cite{Be2}.

For the polygons some results are known. For example we know that
there exists a periodic orbit in all rational polygons(the angles
are rational multiples of $\pi$), and recently Schwartz has proved
in \cite{Schw.05} the existence of a periodic billiard orbit in
every obtuse triangle with angle less than $100$ degrees . A good
survey of what is known about periodic orbits can be found in the
article \cite{Ga.St.Vo} by Galperin, Stepin and Vorobets or in the
book of Masur, Tabachnikov \cite{Ma.Tab}. In this article they
define the notion of stability: They consider the trajectories
which remain periodic if we perturb the polygon. They find a
combinatorial rule which characterize the stable periodic words.
Moreover they find some results about periodic orbits in obtuse
triangles.

The study of the periodic orbits has also been done by famous
physicist. Indeed Glashow and Mittag prove that the billiard
inside a triangle is equivalent to the system of three balls on a
ring, \cite{Gla.Mi}. Some others results can be found in the
article of Ruijgrok and Rabouw \cite{Ra.Rui}. In the polyhedral
case much less is known. The result on the existence of periodic
orbit in a rational polygon can be generalized, but it is less
important, because the rational polyhedra are not dense in the set
of polyhedra. There is no other general result, the only result
concerns the example of the tetrahedron. Stenman \cite{Sten}
shows, that a periodic word of length four exists in a regular
tetrahedron.

The aim of this paper is to find Fagnano's orbit in a regular
tetrahedron and to obtain a rule for the stability of periodic
words in polyhedra. This allows us to obtain a periodic orbit in
each tetrahedron in a neighborhood of the regular one. Moreover we
give examples which prove that the trajectory is not periodic in
all tetrahedra, and we find bounds for the size of the
neighborhood. In the last section we answer to a question of
Galperin, Kruger, Troubetzkoy \cite{Ga.Kr.Tr} by an example of
periodic word $v$ with non periodic points inside its beam.

\section{Statement of results}
The definitions are given in the following sections:

In Section 4 we prove the following result. Consider a periodic biliard orbit coded by the word $v$, then $S_v$ is a certain isometry derived from the combinatorics of the path coded by $v$.
\begin{theorem}
Let $P$ be a polyhedron and $v$ the prefix of a periodic word of
period $|v|$ in $P$.

 If the period is an even number, and $S_v$ is different
from the identity, then $v$ is stable.\\
If the period is odd, then the word is stable if and only if $S_v$
is constant as a function of $P$.
\end{theorem}

In Section 5 we prove:

\begin{theorem}
Assume the billiard map inside the tetrahedron is coded by
$a,b,c,d$.\\
 $\bullet$ The word $abcd$ is periodic for all the
tetrahedra in a neighborhood of the regular one.(This orbit will be denoted as Fagnano's orbit in the following).\\
$\bullet$ In any right tetrahedron Fagnano's orbit does not exist.
There exists an open set of obtuse tetrahedron where Fagnano's
orbit does not exist.
\end{theorem}

The last section of this article is devoted to the study of the
first return map of the billiard trajectory.
\section{Background}\label{pe2}
\subsection{Isometries}
In this part we will recall some usual facts about affine
isometries of $\mathbb{R}^3$. A general reference is \cite{Be2}.

To an affine isometry $a$, we can associate an affine map $f$ and
a vector $u$ such that: $f$ has a fixed point or is equal to the
identity, and such that $a=t_u\circ f=f\circ t_u$ where $t_u$ is
the translation of vector $u$. Then $f$ can be seen as an element
of the orthogonal group $O_3(\mathbb{R})$.
\begin{definition}
First assume that $f$ belongs to $O_3(+)$, and is not equal to the
identity. If $u$ is not an eigenvector of $f$ then $a$ is called
an affine rotation. The axis of $a$ is the set of invariants
points. If $u$ is an eigenvector of $f$, $a$ is called a screw
motion.
In this case the axis of $a$ is the axis of the affine rotation.\\
If $f$, in $O_2(-)$  or $O_3(-)$, is a reflection and $u$ is an
eigenvector of $f$ with eigenvector 1, then $a$ is called a glide
reflection.
\end{definition}

$\bullet$ {\bf Rodrigue's formula.} To finish this subsection we
recall Rodrigue's formula which give the axis and the angle of the
rotation product of two rotations. It can be done by the following
method.
\begin{lemma}\cite{Be3}
We assume that the two rotations are not equal to $Id$, or to a
rotation of angle $\pi$. Let $\theta,u$ the angle and axis of the
first rotation, we denote by $t$ the vector $\tan{\theta/2}.u$ and
$t'$ the associated vector for the second rotation. Then the
product of the two rotations is given by the vector $t''$ such
that
$$t''=\frac{1}{1-t.t'}(t+t'+t\wedge t').$$
\end{lemma}

\subsection{Combinatorics}
\begin{definition}
Let $\mathcal{A}$ be a finite set called the alphabet. By a
language $L$ over $\mathcal{A}$ we mean always a factorial
extendable language: a language is a collection of sets
$(L_n)_{n\geq 0}$ where the only element of $L_0$ is the empty
word, and each $L_n$ consists of words of the form $a_1a_2\dots
a_n$ where $a_i\in\mathcal{A}$ and such that for each $v\in L_n$
there exist $a,b\in\mathcal{A}$ with $av,vb\in L_{n+1}$, and for
all $v\in L_{n+1}$ if $v=au=u'b$ with $a,b\in\mathcal{A}$ then
$u,u'\in L_n$.\\
If $v=a_1a_2\dots a_n$ is a word, then for all $i\leq n$, the word
$a_1\dots a_i$ is called a prefix of $v$.
\end{definition}

\section{Polyhedral billiard}\label{pe3}
\subsection{Definition}
We consider the billiard map inside a polyhedron $P$. This map is
defined on the set $X\subset\partial{P}\times\mathbb{PR}^3$, by
the following method:

$T(m,\theta)=(m',\theta')$ if and only if $mm'$ is colinear to
$\theta$, and $\theta'=S\theta$, where $S$ is the linear
reflection over the face which contains $m'$. The map is not
defined if $m+\mathbb{R}^*\theta$ intersects $\partial{P}$ on an
edge.
$$T:X\rightarrow \partial{P}\times\mathbb{PR}^3$$
$$T:(m,\theta)\mapsto (m',\theta')$$
We identify $\mathbb{PR}^3$ with the unit vectors of
$\mathbb{R}^3$ in the preceding definition.

\subsection{Coding}
We code the trajectory by the letters from a finite alphabet where
we associate a letter to each face.

We call $s_i$ the reflection in the face $i$, $S_i$ the linear
reflection in this face. If we start with a point of direction
$\theta$ which has a trajectory of coding $v=v_{o}..v_{n-1}$ the
image of $\theta$ is: $S_{v_{n-1}}...S_{v_{1}}\theta$. Indeed the
trajectory of the point first meets the face $v_1$, then the face
$v_2$ {\it et caetera}.

If it is a periodic orbit, it meets the face $v_0$ after the face $v_{n-1}$
and we have:\\
 $S_{v_{0}}S_{v_{n-1}}\dots S_{v_{1}}\theta=\theta=S_v\theta$,
$S_v$ is the product of the $S_i$, and $s_v$ the product of the $s_i$.

We recall a result of \cite{Ga.Kr.Tr}: the word $v$ is the prefix
of a periodic word of period $|v|$ if and only if there exists a
point whose orbit is periodic and has $v$ as coding.

\begin{remark}\label{perrme}
If a point is periodic, the initial direction is an eigenvector of
the map $S_v$ with eigenvalue 1. It implies that in $\mathbb{R}^3$, for a
periodic word of odd period, $S$ is a reflection.
\end{remark}

\begin{definition}
Let $v$ be a finite word.
The beam associated to $v$ is the set of
$(m,\theta)$ where $m$ is in the face $v_o$ (resp. edge), $\theta$ a
vector of $\mathbb{R}^3$ (resp. $\mathbb{R}^2$),
such that the orbit of $(m,\theta)$ has
a coding which begins with $v$. We denote it $\sigma_v$.

A vector $u$ of $\mathbb{R}^3$ (resp $\mathbb{R}^2$) is admissible
for $v$, with base point $m$, if there exists a point $m$ in the face (edge)
$v_0$ such that $(m,u)$ belongs to the beam of $v$.
\end{definition}
\begin{lemma}\label{pergeoiso}
Let $s$ be an isometry of $\mathbb{R}^3$ not equal to a
translation. Let $S$ be the associated linear map and $u$ the
vector of translation. Assume $s$ is either a screw motion or a
glide reflection. Then the points $n$ which satisfy
$\overrightarrow{ns(n)}\in\mathbb{R}u,$ are either on the axis of
$s$ (if $S$ is a rotation), or on the plane of reflection. In this
case the vector $\overrightarrow{ns(n)}$ is the vector of the
glide reflection.
\end{lemma}
\begin{proof}
We call $\theta$ the eigenspace of $S$ related to the eigenvalue one.
We have $s(n)=s(o)+S\overrightarrow{on}$ where $o$, the origin of the
base will be chosen later.
Elementary geometry yields
$\overrightarrow{ns(n)}=(S-Id)X+Y$ (where $X=\overrightarrow{on}$,
$Y=\overrightarrow{os(o)}$)
is inside the space $\theta$.

The map $s$ has no fixed point by assumption, thus $\overrightarrow{ns(n)}$
is non-zero.
The condition gives that $(S-Id)X+Y$ is an eigenvector
of $S$ associated to the eigenvalue one. Thus it implies
$$S((S-Id)X+Y)=(S-Id)X+Y,$$
\begin{equation}\label{eqgeoper}
(S-Id)^2X=-(S-I)Y.
\end{equation}
We consider first the case $detS>0$. We choose $o$ on the
axis of $s$.
Then $\theta$ is
a line, we call the direction of the line by the same name.
Since $detS>0$ we have $S\in O_3(+)$ and thus in an appropriate basis
$S$ has the following form
$\begin{pmatrix}
R & 0 \\ 0 & 1
\end{pmatrix},$
where $R$ is a matrix of rotation of $\mathbb{R}^2$. The preceding
equation is equivalent to
$$(R-Id)^2X'=-(R-Id)Y'.$$
where $X'$ is the vector of $\mathbb{R}^2$ such that $X=
\begin{pmatrix}
X' \\ x
\end{pmatrix}$ in this basis.
Furthermore since $S$ is a screw motion with axis
$\begin{pmatrix}0\\0\\1\end{pmatrix}$ in these
coordinates, $Y$ has the following coordinates
$\begin{pmatrix} Y' \\ y\end{pmatrix}$ where $Y'=0$. Since
$S\neq Id$ , $R-Id$ is invertible and thus $X'=0$.
Thus the vectors $X$ solutions of this equation are collinear to the axis.\\

Consider now the case $det(S)<0$, by assumption $S$ is a reflection,
it implies that the eigenspace related to one is a plane. We will solve
Equation \ref{eqgeoper}, we keep notation
$X=\begin{pmatrix}X' \\ x\end{pmatrix}, Y=\begin{pmatrix}Y' \\ y\end{pmatrix}$.

We can assume that $o$ is on the plane of reflection. Moreover we can
choose the coordinates such that that this plane is orthogonal to the line
$\mathbb{R}\begin{pmatrix}0\\0\\1\end{pmatrix}$.
It implies that $S=\begin{pmatrix}1 & 0 &0\\ 0&1&0\\0&0&-1\end{pmatrix},$
and $y=0$. The equation \ref{eqgeoper} becomes $4x=0$. It implies that
$X$ is on the plane of reflection. Since $s$ is a glide reflection, the last
point becomes obvious.
\end{proof}


\begin{proposition}\label{percrucial}
Let $P$ a polyhedron, the following properties are equivalent.\\
(1) A word $v$ is the prefix of a periodic word with period $|v|$.\\
(2) There exists $m\in v_0$ such that
$\overrightarrow{s_{v}(m)m}$ is admissible with base point $m$ for $vv_{0}$,
and $\theta=\overrightarrow{s_{v}(m)m}$ is such that $S\theta=\theta$.
\end{proposition}

\begin{remark}
Assume $|v|$ is even. In the polygonal case the
matrix $S_v$ can only be the identity, thus $s_v$ is a translation.
We see by unfolding that $s_v$ can not have a fixed point, thus in the
polyhedral case $s_v$ is either a translation or a screw motion or a
glide reflection.
If we do not assume the admissibility in condition $(2)$ it is not
equivalent to condition ($1$) as can be seen in a obtuse triangle, or a
right prism above the obtuse triangle and the word $abc$.
\end{remark}
\begin{figure}[hbt]
\begin{center}
\includegraphics[width= 4cm]{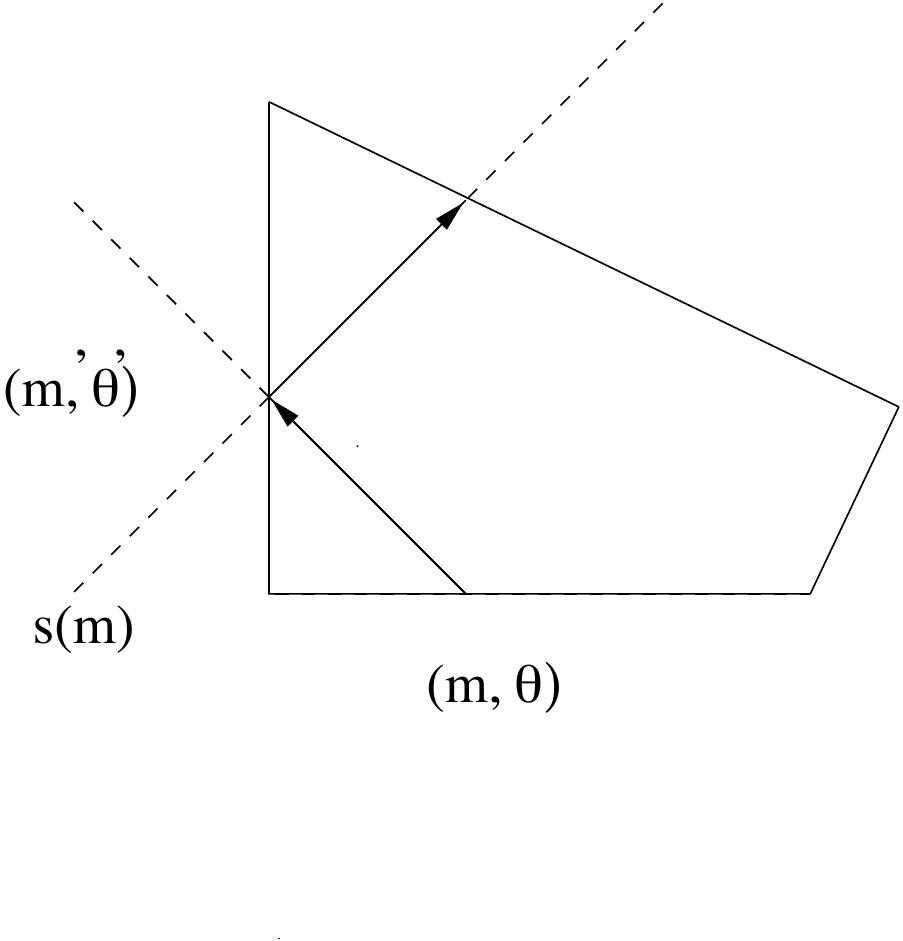}
\caption{Billiard orbit and the associated map}
\label{perfig2}
\end{center}
\end{figure}
 \begin{proof}

First we claim the following fact. The vector connecting
$T^{|v|}(m,\theta)$ to $s_{v}(m)$ is parallel to the direction of
$T^{|v|}(m,\theta)$. For $|v|=1$ if the billiard trajectory goes from
$(m,\theta)$ to $(m',\theta')$ without
reflection between, then the direction $\theta'$ is parallel to
$\overrightarrow{s(m)m'}$, where $s$ is the reflection over the face of
$m'$ see Figure \ref{perfig2}. Thus the claim follows combining this
observation with an induction
argument.\\

Next assume (1). Then there exists $(m,\theta)$ periodic. We deduce that
$S\theta=\theta$, moreover this direction is admissible.
Then the claim implies that $\overrightarrow{s_v(m)m}=\theta$ and thus is
admissible for $vv_{0}$.\\
Finally assume (2).
First we consider the case where $S\neq Id$.
Lemma \ref{pergeoiso} implies that $m$ is on the axis of $s$ if $|v|$ is even,
otherwise on the plane of reflection.
If $|v|$ is even then $\theta=\overrightarrow{s(m)m}$ is collinear to the
axis of the screw motion. Since we have assumed $\overrightarrow{s_v(m)m}$
admissible we deduce that $\theta$ is admissible with base point $m$.
If $|v|$ is odd then Lemma \ref{pergeoiso} implies that $\theta$ is the direction
of the glide.
The hypothesis implies that $\theta$ is admissible for $v$.\\
Now we prove that $(m,\theta)$ is a periodic trajectory.
We consider the image $T^{|v|}(m,\theta)$. We denote this point
$(p,\theta')$.
We have by hypothesis that $p$ is in $v_0$.
The above claim implies that $\overrightarrow{s_v(m)p}$ is parallel to
the direction $\theta'$. The equation $S\theta=\theta$ gives
$\theta'=\theta$. Thus we have $\overrightarrow{s_v(m)m}$ is parallel to
$\overrightarrow{s_v(m)p}$, since we do not consider direction included in
a face of a polyhedron this implies $p=m$. Thus $(m,\theta)$ is
a periodic point.

If $S=Id$, then $s$ is a translation of vector $\overrightarrow{s_v(m)m}=u$.
The vector $u$ is admissible. Then we consider a point $m$ on the
face $v_0$ which is admissible. Then we show that $(m,u)$ is a periodic
point by the same argument related to the claim.
\end{proof}

Thus we have a new proof of a result of \cite{Ga.Kr.Tr}:
\begin{theorem}
Let $v$ be a periodic word of even length. The set of periodic
points in the face $v_0$ with code $v$ and length $|v|$ can have
two shapes. Either it is an open set or it is a point.

If $v$ is a periodic word of odd length, then the set of periodic points in
the face $v_0$ with code $v$ and period $|v|$ is a segment.
\end{theorem}
\begin{proof}
Let $\Pi$ be a face of the polyhedron, and let $m\in \Pi$ be the starting point for a periodic billiard path. The first return map to $m$ is an isometry of $\mathbb{R}^3$ that fixes both $m$ and the direction $u$ of the periodic billiard path. 

Assume first $|v|$ is odd. Then the first return map is a reflection since it fixes a point. 
Then it fixes a plane $\Pi'$. Note that $u\in \Pi'$, and that the intersection $\Pi\cap\Pi'$ is a segment. Points in this segment sufficiently near $m$ have a periodic orbit just as the one starting at $v$.

Assume now $|v|$ is even, we will use Proposition \ref{percrucial}.
If $S_v$ is the identity, then the periodic points are the
points such that the coding of the billiard orbit in the direction of the
translation begins with $v$,
otherwise there is a single point, at the intersection of the axis of
$s$ and $v_0$. However the set of points with code $v$ is still an open set.

\end{proof}
Moreover our proof gives an algorithm to locate this set in the face.
We will use it in Section \ref{pe5}.

\section{Stability}\label{pe4}

\subsection{Notations and definitions} First of all we define the
topology on the set of polyhedra with $k$ vertices. As in the
polygonal case we identify this set with $\mathbb{R}^{3(k-2)}$.
But we remark the following fact\!: Consider a polyhedron $P$ such
that a face of $P$ is not a triangle. Then we can find a
perturbation of $P$, as small as we want, such that the new
polyhedron has a different combinatorial type ({\it i.e} the
number of vertices, number of edges, or number of faces is different).\\
In this case consider a triangulation of each face which does not
add new vertices. Consider the set of all such triangulations of
all faces. There are finitely many such triangulations. Each can
be considered as a combinatorial type of the given polyhedron. Let
$B(P,\varepsilon)$ be the ball of radius $\varepsilon$ in
$\mathbb{R}^{3(k-2)}$ of polyhedra $Q$. If $P$ has a single
combinatorial type, $\varepsilon$ is chosen se small that all $Q$
in the ball have the same combinatorial type. If $P$ has several
combinatorial types, then $\varepsilon$ is taken so small that all
$Q$ have one of those combinatorial type. The definition of
stability is now analogous to the definition in polygons.

In an other way let $v$ be a periodic word in $P$ and $g$ a
piecewise similarity. Consider the polyhedron $g(P)$, and the same
coding as in $P$. If $v$ exists in $g(P)$ it is always a periodic
word in $g(P)$. We note that the notion of periodicity only
depends on the normal vectors to the planes of the faces.

\subsection{Theorem}
\begin{theorem}\label{perstable}
Let $P$ be a polyhedron and $v$ the prefix of a periodic word of
period $|v|$
in $P$.\\
If the period is even, and $S_v$ is different from the identity,
then $v$ is stable.\\
If the period is odd, then the word is stable if and only if $S_v$ is constant
as a function of $P$.
\end{theorem}
\begin{remark}
The second point has no equivalence in dimension two, since each
element of $O(2,-)$ is a reflection. It is not the case for
$O(3,-)$.
\end{remark}
\begin{proof}
First consider the case of period even. The matrix $S=S_v$ is not
the identity, and $\theta=\theta_v$ is the eigenvector associated
to the eigenvalue one. First note that by continuity $v$ persists
for sufficiently small perturbations of the polyhedron. Fix a
perturbation and let $B=S_v^Q$ be the resulting rotation for the
new polyhedron $Q$. We will prove that the eigenvalue of $S$ is a
continuous function of $P$. We take the reflections which appears
in $v$ two by two. The product of two of those reflections is a
rotation. We only consider the rotations different of the
identity. The axes of the rotations are continuous map as function
of $P$ since they are at the intersection of two faces. Then
Rodrigue's formula implies that the axes of the rotation, product
of two of those rotations, are  continuous maps of the polyhedron,
under the assumption that the rotation is not the identity
(because $t$ must be of non-zero norm). Since $S^P$ is not equal
to $Id$, there exists  a neighborhood of $P$ where $S^Q\neq Id$.
It implies that the axis of $S^P$ is a continuous function of $P$.
Thus the two eigenvectors of $B,S$ are near if $B$ is sufficiently
close to $S$. The direction $\theta$ was admissible for $v$, we
know that the beam of $v$ is an open set of the phase space
\cite{Ga.Kr.Tr}, so we have for $Q$ sufficiently close to $P$ that
$\alpha$ (the real eigenvector of $B$) is admissible for the same
word. Moreover the foot points are not far from the initial points
because they are on the axis of the isometries.
Thus the perturbated word is periodic by Proposition \ref{percrucial}.\\

If the length of $v$ is odd, then Remark \ref{perrme} implies that
$S$ is a reflection. We have two cases\!: Either $S_v$ is constant, or not.
If it is not a constant function, then in any neighborhood there exists
a polyhedron $Q$ such that $S_v^Q$ is different from a reflection.
Then the periodic trajectory can not exist in $Q$.
If $S_v$ is constant, then it is always a reflection, and a similar argument to
the even case shows that the plane of reflection of $S$ is a continuous map of
$P$. It finishes the proof.
\end{proof}

\begin{corollary}We have three consequences
\begin{itemize}
\item All the words of odd length are stable in a polygon.
\item Consider a periodic billiard path in a right prism. Then its projection inside the polygonal basis is a billiard path. We denote the coding of the projected trajectory as the projected word.
Assume that the projected word is not stable in the polygonal basis. Then the word is unstable.
\item All the words in the cube are unstable.
\end{itemize}
\end{corollary}
The first point was already mentioned in \cite{Ga.St.Vo}.

\begin{proof}
For the first point the proof is the same as the proof of the theorem.
Indeed is $|v|$ is odd then $s$ has a real eigenvector, and we can apply the
proof.\\
For the second point we begin with the period two trajectory which hits the top and the bottom of the prism. It is clearly unstable, for example we can change one face and keep the other. Let $v$ be any other periodic word, and $w$ the word
corresponding to the projection of $v$ to the base of the prism assumed to be
unstable. We perturb a vertical face of the prism
such that this face contains an edge which appears in the coding of $w$.
The word $v$ can not be periodic in this polyhedron by unstability of $w$.\\
For the cube, let $v$ be a periodic word, by preceding point its projection on each coordinate plane must be stable. But an easy computation shows that no word is stable in the square.
\end{proof}

We remark that the two and three dimensional cases are different for the
periodic trajectories of odd length. They are all stable in one case, and all
unstable in the second.
Recently Vorobets has shown that if $S_v=Id$ then the word is not stable
\cite{Vo.04}.

\section{Tetrahedron}\label{pe5}
In the two following Sections we prove the following result.
\begin{theorem}\label{tetra}
Assume the billiard map inside the tetrahedron is coded by
$a,b,c,d$.\\

 $\bullet$ The word $abcd$ is periodic for all the
tetrahedra in a neighborhood of the regular one.\\

$\bullet$ In any right tetrahedron Fagnano's orbit does not exist.
There exists an open set of obtuse tetrahedron where Fagnano's
orbit does not exist.
\end{theorem}
\begin{remark}
Steinhaus in his book \cite{Stein}, cites Conway for a proof that
$abcd$ is periodic in all tetrahedra, but our theorem gives a
counter example. Moreover our proof gives an {\bf algorithm} which
find the coordinates of the periodic point, when it exists.
\end{remark}
For the definition of obtuse tetrahedron, see Section 6.
\subsection{Regular tetrahedron}
We consider a regular tetrahedron. We can construct a periodic
trajectory of length four, which is the generalization of
Fagnano's orbit. To do this we introduce the appropriate coding
(see Figure \ref{perfig4}).
\begin{figure}[hbt]
\begin{center}
\includegraphics[width= 4cm]{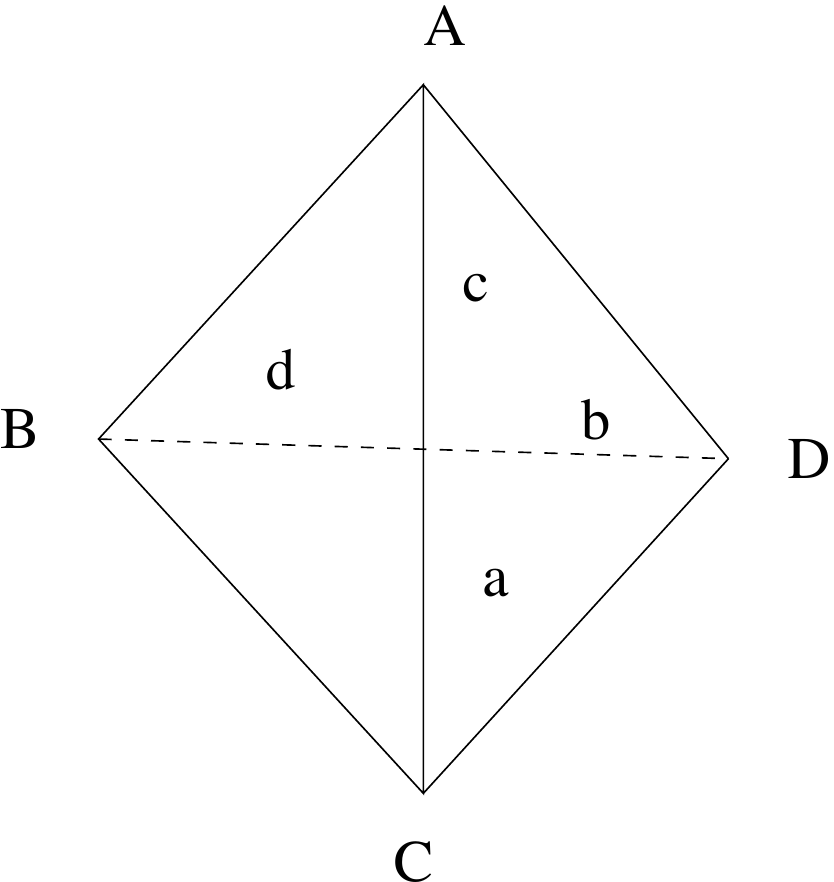}
\caption{Coding of a tetrahedron}
\label{perfig4}
\end{center}
\end{figure}
In this figure, the letter $a$ is opposite to the vertex $A$, etc. 
\begin{lemma}
Let $ABCD$ be a regular tetrahedron, with the natural coding. If
$v$ is the word $adcb$, there exists a direction $\theta$, there
exists an unique point $m$ such that $(m,\theta)$ is periodic and
has $v$ as prefix of its coding. Moreover $m$ is on the altitude
of the triangle $BCD$ which starts at $C$.
\end{lemma}
\begin{remark}
If we consider the word $v^n$, the preceding point $m$ is the unique
periodic point for $v^n$. Indeed the map $s_{v^n}$ has the same axis as $s_v$,
and we use Proposition \ref{percrucial}.

We use the following coordinates for two reasons. First these coordinates were used by  Ruijgrok and Rabouw \cite{Ra.Rui}. Secondly with these coordinates the matrix $S_v$ has rational entries, and the computations seems more simples.
\end{remark}
\begin{proof}
The lemma has already been proved in \cite{Sten}, but we rewrite it
in a different form with the help of Proposition \ref{percrucial}.

We have $S_v=S_a.S_b.S_c.S_d=R_{DC}.R_{AB}$ where $R_{DC}$ is the linear
rotation of axis $DC$, it is a product of the two reflections.
We compute the real eigenvector of $S_v$, and we obtain the point $m$ at
the intersection of the axis of $s$ and the face $BCD$.
We consider an orthonormal base of $\mathbb{R}^3$ such that the points
have the following coordinates, see \cite{Sten}\!:
$$A=\frac{\sqrt{2}}{4}\begin{pmatrix}-1\\-1\\-1\end{pmatrix},\quad
D=\frac{\sqrt{2}}{4}\begin{pmatrix}1\\-1\\1\end{pmatrix},$$
$$\quad C=\frac{\sqrt{2}}{4}\begin{pmatrix}1\\1\\-1\end{pmatrix},\quad
B=\frac{\sqrt{2}}{4}\begin{pmatrix}-1\\1\\1\end{pmatrix}.$$
The matrices of $S_a,S_d,S_c,S_b$ are\!:
$$\frac{1}{3}\begin{pmatrix}
1 & -2 & -2\\
-2 & 1 & -2\\
-2 & -2 & 1\\
\end{pmatrix},
\frac{1}{3}\begin{pmatrix}
1 & 2 & 2\\
2 & 1 & -2\\
2 & -2 & 1\\
\end{pmatrix},
\frac{1}{3}\begin{pmatrix}
1 & -2 & 2\\
-2 & 1 & 2\\
2 & 2 & 1\\
\end{pmatrix},
\frac{1}{3}\begin{pmatrix}
1 & 2 & -2\\
2 & 1 & 2\\
-2 & 2 & 1\\
\end{pmatrix}.
$$
Now we obtain $S$.
$$S=S_aS_bS_cS_d,$$
$$S=\frac{1}{81}\begin{pmatrix}
-79 & -8 & 16\\
8 & 49 & 64\\
-16 & 64 & -47\\
\end{pmatrix}$$
The real eigenvector is
$u=\frac{1}{\sqrt{5}}\begin{pmatrix}0\\2\\1\end{pmatrix}$.
Now we compute the vector $N$ such that $s(X)=SX+N$.
To do this we use the relation $s(A)=s_a(A)$.
$s_a$ is the product of $S_a$ and a translation of vector $v$.
We obtain
$$v=\frac{\sqrt{2}}{6}\begin{pmatrix}1\\1\\1\end{pmatrix},$$
$$s(A)=\frac{5\sqrt{2}}{12}\begin{pmatrix}1\\ 1\\ 1\end{pmatrix},$$
 $$N=\frac{\sqrt{2}}{81}\begin{pmatrix}
16\\64\\34\\ \end{pmatrix}.$$

We see that $s$ is a screw motion.
Finally we find the point at the intersection of the axis and the face $a$.
The points of the axis verify the equation
\begin{equation*}
SX+N=X+\lambda u.
\end{equation*}
where $X$ are the coordinates of the point of the axis, and $\lambda$ is
a real number.

The point $m$ is on the face $a$ if we have
\begin{equation*}
(\overrightarrow{Cm}|\overrightarrow{CB}\wedge\overrightarrow{CD})=0.
\end{equation*}
So $X$ is the root of the system made by those two equations.
The last equation gives $x+y+z=\frac{\sqrt{2}}{4}$.
We obtain
$$m=\frac{\sqrt{2}}{20}\begin{pmatrix}2 \\ 2 \\ 1\end{pmatrix}.$$
We remark that $(\overrightarrow{Cm}|\overrightarrow{DB})=0$ which proves
that $m$ is on the altitude of the triangle $BCD$.
\end{proof}
In fact there are six periodic trajectories of length four, one for each of the
word
$$abcd,abdc,acbd,acdb,adbc,adcb.$$
The six orbits come in pairs which are related by the natural
involution of direction reversal. Now we can ask the same question
in a non regular tetrahedron. Applying Theorem \ref{perstable}
yield the first part of Theorem \ref{tetra}.

Now the natural question is to characterize the tetrahedron which contains
this periodic word.

\section{Stability for the tetrahedron}\label{pe7}
\begin{definition}
A tetrahedron is acute if and only if
in each face the orthogonal projection of the other vertex is
inside the triangle.

A tetrahedron is right if and only if there exists a vertex, where the three
triangles are right triangles.
\end{definition}
We recall that an acute triangle is a triangle where all the angles are
less than $\frac{\pi}{2}$.
For the polygons our definition is equivalent to the natural definition.

\subsection{Proof of second part of Theorem \ref{tetra}}
We consider a tetrahedron $ABCD$ with vertices
$$A=(0,0,0)\quad B=(a,0,0)\quad C=(0,b,0)\quad D=(0,0,1).$$
We study the word $v=abcd$.

We have $S=S_a*S_d*S_c*S_b$.

 Since
$$S_b\begin{pmatrix}-1 & 0 & 0 \\ 0 & 1 & 0 \\ 0 & 0 & 1\end{pmatrix},$$
$$S_c=\begin{pmatrix} 1 & 0 & 0 \\ 0 & -1 & 0 \\ 0 & 0 & 1\end{pmatrix},$$
$$S_d=\begin{pmatrix} 1 & 0 & 0 \\ 0 & 1 & 0 \\ 0 & 0 & -1\end{pmatrix},$$

we obtain $S=-S_a$, thus $S$ has 1 for eigenvalue, and the
associated eigenvector is the normal vector
to the plane $a$. We remark that $s(A)=s_a(A)$.
The fact that $S=-S_a$ implies that $S$ is a rotation of angle $\pi$, thus
$s$ is the product of a rotation of angle $\pi$ and a translation.

Consider the plane which contains $A$ and orthogonal to the axis of $S$,
let $O$ the point of intersection. Then $S$ is a rotation of angle $\pi$, thus
 $O$ is the middle of $[AE]$, where $E$ is given by
$S(\overrightarrow{OE})=\overrightarrow{OA}$.
It implies that the middle $M$ of the edge $[As(A)]$ is on the axis of $s$, see
Figure \ref{pef6}.
\begin{figure}[h]
\begin{center}
\includegraphics[width= 4cm]{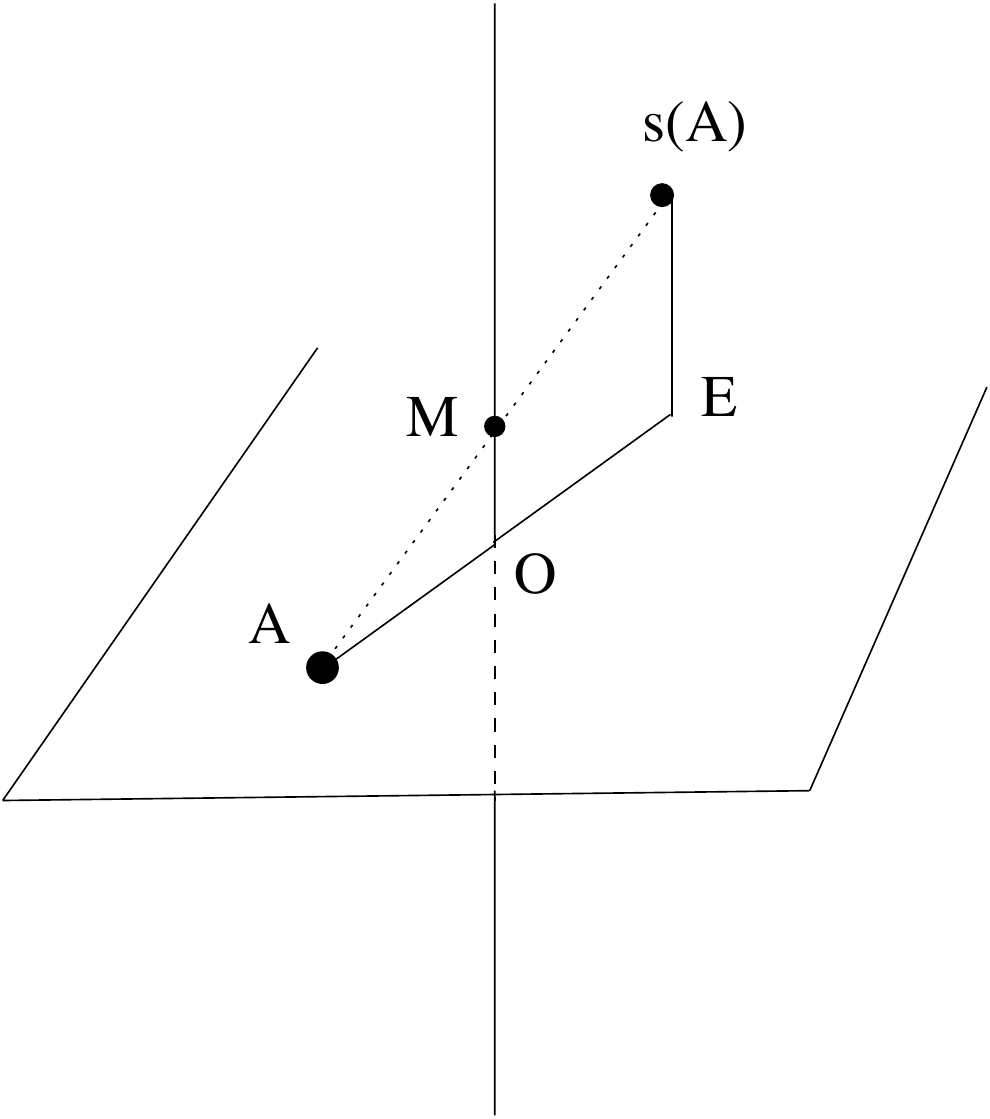}
\caption{Screw motion associated to the word abcd}
\label{pef6}
\end{center}
\end{figure}
Clearly $m$ is a point in the side $ABC$. If $v$ is periodic then applying
Proposition \ref{percrucial} yields that $M$ is the base point of the periodic
trajectory.
Moreover since the direction of the periodic trajectory is the normal vector
to the plane $a$, we deduce that $A$ is on the trajectory.
So the periodic trajectory cannot exist.

Now we prove the second part of the theorem.
We give an example of obtuse tetrahedron where Fagnano's orbit does not exist.

In this example the point on the initial face, which must be
periodic see Proposition \ref{percrucial}, is not in the interior of the triangle.

We consider the tetrahedron $ABCD$
$$A(0,0,0)\quad B(2,0,0)\quad C(1,1,0)\quad D(3,2,1).$$

We study the word $v=abcd$.

We obtain the matrix of $S_v$
$$\begin{pmatrix}
1/33 & 8/33 & 32/33 \\
104/165 & -25/33 & 28/165 \\
128/165 & 20/33 & -29/165
\end{pmatrix}.$$
Now $s$ is the map $SX+N$ where
$N=\begin{pmatrix} 4/11 \\ 4/11 \\- 12/11\end{pmatrix}$.

$S$ has the following eigenvalue
$u=\begin{pmatrix} 9/8 \\ 1/2 \\ 1\end{pmatrix}$.

Now $s$ is a screw motion and we find the point at the intersection
of the axis of $s$ and the face $a$ \!:
We must solve the system
\begin{gather}
Sm+N=m+\lambda.u \\
(\overrightarrow{Bm}.n)=0
\end{gather}
It is equivalent to the system
$$\begin{pmatrix} S-Id & -u\\n^t & 0\end{pmatrix}
\begin{pmatrix}m\\ \lambda\end{pmatrix}=
\begin{pmatrix}-N\\2\end{pmatrix}.$$
where $n$ is the normal vector to the face $BCD$.
$$n=\begin{pmatrix}1 \\ 1 \\ -3\end{pmatrix}.$$
We obtain the matrix
$$\begin{pmatrix}
-32/33 & 8/33 & 32/33 & -9/8 \\
104/165 & -58/33 & 28/165 & -1/2 \\
128/165 & 20/33 & -194/165 & -1 \\
1 & 1 & -3 & 0
\end{pmatrix}.$$
We obtain $m$\!:
$$m=\begin{pmatrix}
22/161 \\ 6/23 \\ -86/161
\end{pmatrix}.$$
But this point is not inside $BCD$.
Moreover we see that this point is not on
the altitude at $BD$ which passes through $C$.

The tetrahedron is obtuse, due to the triangle $ABD$.
The triangle $BCD$ is acute, and the axis of $s$ does not cut this face
in the interior of the triangle.

Moreover we obtain that there exists a neighborhood of this tetrahedron, where
Fagnano's word is not periodic.
Indeed in a neighborhood the point $m$ can not be in the interior of $ABCD$.
\begin{remark}
We can remark that our proof gives a criterion for the existence of a periodic billiard path of this type. One computes the axis of the screw motion, and find if it intersects the relevant faces. 

For a generic tetrahedron we can use it to know if there exists a Fagnano's orbit. But we have not find a good system of coordinates where the computations are easy. Thus we are not able to caracterize the tetrahedra with a Fagnano's orbit.
\end{remark}
\section{First return map}
 In this section we use the preceding example to study a related problem for periodic billiard paths. We answer to a question of
Galperin, Kruger, Troubetzkoy \cite{Ga.Kr.Tr} by an example of
periodic word $v$ with non periodic points inside its beam.

We consider the word $v=(abcd)^\infty$ and the set $\sigma_v$. The projection of this set on the face
$a$ is an open set. Each point in this open set return to
the face $a$ after three reflections. We study this return map and the set $\pi_{a}(\sigma_v)$. We consider
the same basis as in Section 6.
Moreover, in the face $a$ we consider the following basis
$$\begin{pmatrix}\frac{\sqrt{2}}{4}\\0\\0\end{pmatrix}+\mathbb{R}\begin{pmatrix}1\\0\\-1\end{pmatrix}+\mathbb{R}
\begin{pmatrix}1\\-2\\1\end{pmatrix}.$$
\begin{theorem}
In the regular tetrahedron, consider the word $v=(abcd)^\infty$.
Then the set $\pi_{a}(\sigma_v)$ is an open set. There exists only
one point in this set with a periodic billiard orbit.
\end{theorem}
The theorem of \cite{Ga.Kr.Tr} explains that some such cases could
appear, but there were no example before this result.

 Theorem 11 means that for all point
in $\pi_{a}(\sigma_v)$, except one, the billiard orbit is coded by
a periodic word, but it is never a periodic trajectory. For the
proof we begin by the following lemma.
 \begin{lemma}
In the regular tetrahedron, consider the word $v=(abcd)^\infty$. The first return map $r$ on
$\pi_{a}(\sigma_v)$ has the following equation
$$r: V=\displaystyle\begin{pmatrix}x\\y\end{pmatrix}\mapsto AV+B,$$
$$A=\frac{1}{81}\displaystyle\begin{pmatrix}-83&28\\-12&-75\end{pmatrix}, B+\frac{1}{81}\displaystyle
\begin{pmatrix}-15\\9\end{pmatrix}.$$

The set $\pi_{a}(\sigma_v)$ is the interior of the biggest ellipse of center $m$ related to the matrix $A$.
\end{lemma}
\begin{proof}
If $m$ is a point of the face $a$, the calculus of Section 6 shows that
$$rm=\frac{1}{81}\begin{pmatrix}
-79x -8y +16\sqrt{2}\\
66x-21y+42z+\frac{3\sqrt{2}}{2}\\
13x+29y-58z\frac{33\sqrt{2}}{12}\\
\end{pmatrix}$$
Now we compute $m$ and $rm$ in the basis of the face $a$. We
obtain the matrices $A,B$.
\end{proof}
\subsection{Proof of Theorem 11}
We can verify that the periodic point
$\frac{\sqrt{2}}{20}\begin{pmatrix}2\\2\\1\end{pmatrix}$ is fixed
by $r$. Indeed in this basis, it becomes
$\frac{\sqrt{2}}{20}\begin{pmatrix}-2\\1\end{pmatrix}$. Now the
orbit of a point under $r$ is contained on an ellipse related to
the matrix $A$. This shows that the set $\pi_{a}(\sigma_v)$ is the
biggest ellipse included in the triangle. And an obvious
computation shows that only one point is fixed by $r$.
\bibliography{docbib}
\end{document}